\documentclass[12pt]{amsart}
\usepackage{amsmath, amsthm, amssymb, enumitem, mathtools, tikz, tkz-euclide, xcolor, thmtools}
\usepackage{thm-restate}
\usepackage[symbol, hang, flushmargin]{footmisc}
\usepackage[margin=1.5in]{geometry}
\usepackage[colorlinks=false]{hyperref}

\newtheorem{theorem}{Theorem}[section]
\newtheorem{definition}[theorem]{Definition}

\newtheorem{lemma}[theorem]{Lemma}
\newtheorem{claim}{Claim}[theorem]

\newtheorem{corollary}[theorem]{Corollary}
\newtheorem{conjecture}[theorem]{Conjecture}

\newcommand{\phiset}[3]{\phi_{{#1}\rightarrow{#2}}({#3})}
\newcommand{\W}[2]{W_{#1}^{#2}}

\newcommand{\C}{\mathcal{C}}
\newcommand{\cC}{\mathcal{C}}

\newcommand{\bZ}{\mathbb{Z}}
\newcommand{\del}{\setminus}

\begin{document}
\sloppy
	
\title{The grid theorem for vertex-minors}

\author{Jim Geelen}
\address{Department of Combinatorics and Optimization, University of Waterloo}
\author{O-Joung Kwon}
\address{\begin{tabular}[t]{@{}l@{}}
Department of Mathematics, Incheon National University \\ Discrete Mathematics Group, Institute for Basic Science (IBS)
\end{tabular}}
\author{Rose McCarty}
\address{Department of Combinatorics and Optimization, University of Waterloo}
\author{Paul Wollan}
\address{Department of Computer Science, University of Rome, ``La Sapienza''}

\thanks{This research was partially supported by a grant from the
Office of Naval Research [N00014-10-1-0851] and NSERC [203110-2016]. The second author was supported by the National Research Foundation of Korea (NRF) grant funded by the Ministry of Education (No. NRF-2018R1D1A1B07050294), and by the Institute for Basic Science (IBS-R029-C1).}

\date{\today}

\begin{abstract}
We prove that, for each circle graph $H$, every graph with sufficiently large rank-width contains a vertex-minor isomorphic to $H$.
\end{abstract}

\maketitle

%INTRO
\section{Introduction}
\label{intro}
We prove the following result.

\begin{theorem}
\label{main}
For each circle graph $H$, there is an integer $r(H)$ so that every graph with no vertex-minor isomorphic to $H$ has rank-width at most $r(H)$.
\end{theorem}

\noindent We define {\em circle graphs} in Section~\ref{sec:circleGraphs}, and we define {\em rank-width} and {\em vertex-minors}  in Section~2. 

For any fixed circle graph $H$, Theorem~\ref{main} gives a polynomial-time algorithm for testing for a vertex-minor isomorphic to $H$. Jeong, Kim, and Oum~\cite{JKO} provided an efficient algorithm that, for a given graph $G$, determines whether or not the rank-width of $G$ is at most $r(H)$, and, if the rank-with is at most $r(H)$, finds a rank-decomposition of width at most $r(H)$. By Theorem~\ref{main}, we may assume that the rank-width of $G$ is at most $r(H)$, as otherwise $G$ has a vertex-minor isomorphic to $H$. Then, using the rank-decomposition of width at most $r(H)$ for $G$, we can determine whether or not $G$ has a vertex-minor isomorphic to $H$ via dynamic programming~\cite{WeakSeeseConjPf}; for further details see the survey~\cite{RWSurvey}.
\begin{corollary}
\label{vm-testing}
For each circle graph $H$ there is a polynomial-time algorithm that tests, for any given graph $G$, whether or not $G$ contains a vertex-minor isomorphic to $H$.
\end{corollary}

Using a result of Dvo\v{r}\'ak and Kr\'al'~\cite{RWChiBounded}, Theorem~\ref{main} also implies that for every circle graph $H$, the class of graphs with no vertex-minor isomorphic to $H$ is \textit{$\chi$-bounded} (the chromatic number of each graph in the class is bounded by a function of its clique number). The first author conjectured that for every graph $H$, the class of graphs with no vertex-minor isomorphic to $H$ is $\chi$-bounded; Davies~\cite{davies20} very recently proved this conjecture. The next step, as proposed by Kim, Kwon, Oum and Sivaraman~\cite{Kim2020}, is to determine if, for every graph $H$, the class of graphs with no vertex-minor isomorphic to $H$ is \textit{polynomially $\chi$-bounded} (there exists a polynomial $\chi$-bounding function). Our result and a recent theorem of Bonamy and Pilipczuk~\cite{Bonamy19} imply that this is true when $H$ is a circle graph.

\newtheorem*{gridthm}{Grid Theorem}

Theorem~\ref{main} is analogous to the Grid Theorem of Robertson and Seymour~\cite{graphMinors5}, stated below.
\begin{gridthm}
For each planar graph $H$, there is an integer $t$ so that every graph with no minor isomorphic to $H$ has tree-width at most $t$.
\end{gridthm}

Since each planar graph is isomorphic to a minor of some grid, it suffices to prove the Grid Theorem when $H$ is itself a grid. For vertex-minors the role of grids is assumed by ``comparability grids".

For a positive integer $n$, the $n \times n$ \textit{comparability grid} is the graph with vertex set $\{(i,j): i,j \in \{1,2,\ldots, n\}\}$ where there is an edge between vertices $(i,j)$ and $(i',j')$ if either $i\leq i'$ and $j\leq j'$, or $i\geq i'$ and $j\geq j'$. Every circle graph is isomorphic to a vertex-minor of a comparability grid (see Lemma~\ref{lemma:compGrid}), so it suffices to prove Theorem~\ref{main} when $H$ is itself a
comparability grid. Thus Theorem~\ref{main} is equivalent to the following result.
\begin{restatable}{theorem}{mainG}
\label{mainGrid}
There is a function $f:\bZ\rightarrow\bZ$ so that for every positive integer $n$, every graph of rank-width at least $f(n)$ has a vertex-minor isomorphic to the $n \times n$ comparability grid.
\end{restatable}

Despite the resemblance, we see no way of directly proving the Grid Theorem from Theorem~\ref{main} or vice-versa. However, the following conjecture of Oum~\cite{lineGraphsGridThm} about pivot-minors (defined in Section~2), if true, would imply both results.
\begin{conjecture}
\label{conj:pivot}
For each bipartite circle graph $H$, there is an integer $r$ so that every graph with no pivot-minor isomorphic to $H$ has rank-width at most $r$.
\end{conjecture} 
\noindent Oum's conjecture would imply Theorem~\ref{main} because every pivot-minor of a graph $G$ is also a vertex-minor of $G$ by definition, and, more importantly, because every circle graph is a vertex-minor of a bipartite circle graph~\cite[Corollary~53]{BrijderTraldi16}. The conjecture is known to hold for bipartite graphs, as that special case is equivalent to the grid theorem for binary matroids; see~\cite{GridThmRepresentable}. Oum~\cite{lineGraphsGridThm} also proved Conjecture~\ref{conj:pivot} for line graphs and circle graphs. It is natural to ask if something similar could hold for induced subgraphs, but this is unlikely; see~\cite{infinitelyManyCW}.

The main new tool in our proof of Theorem~\ref{main} is a ``disentangling lemma'', Lemma~\ref{shortPath}. This result is particular to vertex-minors; it does not extend to pivot-minors and there is no analogue for minors. We also rely on a recent theorem of Kwon and Oum~\cite{scattered} (stated in this paper as Theorem~\ref{thm:SF}) to serve as the base case for induction. 

%BACKGROUND
\section{Preliminaries}
\label{sec:background}

All graphs in this paper are finite and simple; for a graph $G=(V,E)$ we consider $E$ as a set consisting of unordered pairs of vertices. The set of neighbours of a vertex $v$ in a graph $G$ is denoted by $N(v)$.

In this section we review some material on vertex-minors, pivot-minors, and rank-width; these results  are mostly due to Bouchet~\cite{graphicIsoSystems} and Oum~\cite{RWAndVM}.

\subsection*{Vertex-minors and pivot-minors}
For a vertex $v$ of a graph $G$, we write $G* v$ for the graph formed from $G$ by replacing the induced subgraph of $G$ on the set of neighbours of $v$ with its complement. We say that $G*v$ is obtained from $G$ by \textit{local complementation at $v$}. A graph $H$ is a \textit{vertex-minor} of $G$ if $H$ can be obtained from $G$ by a sequence of vertex deletions and local complementations. If $H$ can be obtained from $G$ by local complementations only, then we say that $H$ and $G$ are \textit{locally equivalent}. Note that,
if a graph $H$ is a vertex-minor of a graph $G$, then there exists a graph locally equivalent to $G$ that has $H$ as an induced subgraph.

For an edge $uv$ of a graph $G$, we write $G \times uv$ for the graph $G*u*v*u$. We say that $G\times uv$ is obtained from $G$ by \textit{pivoting on $uv$}. The graph $G \times uv$ is well-defined since $G*u*v*u = G*v*u*v$; see~\cite[Corollary~2.2]{RWAndVM}. A graph $H$ is a \textit{pivot-minor} of $G$ if $H$ can be obtained from $G$ by a sequence of vertex deletions and pivots. We use pivoting extensively and the following result explicitly describes the effects of the operation.

\begin{lemma}\cite[Proposition~2.1]{RWAndVM}
\label{lemma:pivot}
Let $uv$ be an edge of a graph $G$ and let $F$ denote the set of all unordered pairs $xy$ where $x$ and $y$ are in distinct parts of
$(N(u)\cap N(v),\, N(u)\setminus (N(v)\cup\{v\}),\, N(v)\setminus (N(u)\cup\{u\}))$. 
Then $G \times uv$ is the graph formed from $G$ by first replacing its edge set with the symmetric difference of $E(G)$ and $F$ and then switching the labels of the vertices $u$ and $v$.
\end{lemma}

The following results show that there are two ways to remove a vertex with respect to pivot-minors and three ways with respect to vertex-minors; see~\cite{graphicIsoSystems}. 
\begin{lemma}
\label{lemma:pivot2}
Let $H$ be a pivot-minor of a graph $G$. If $v\in V(G)\setminus V(H)$, then either
\begin{itemize}
\item $H$ is a pivot-minor of $G-v$, or
\item for each $w\in N(v)$,  the graph $H$ is a pivot-minor of $(G\times vw)-v$.
\end{itemize}
\end{lemma}

\begin{lemma}
\label{lemma:vm}
Let $H$ be a vertex-minor of a graph $G$. If $v\in V(G)\setminus V(H)$, then either
\begin{itemize}
\item $H$ is a vertex-minor of $G-v$, 
\item $H$ is a vertex-minor of $(G*v)-v$, or
\item for each $w\in N(v)$,  the graph $H$ is a vertex-minor of $(G\times vw)-v$.
\end{itemize}
\end{lemma}

Underlying these two results is the fact that, if $v$ is a vertex in a graph $G$ and $w_1$ and $w_2$ are neighbours of $v$, then
$(G\times vw_1)-v$ and $(G\times vw_2)-v$ are equivalent up to pivoting. By Lemma~\ref{lemma:pivot}, the vertices $w_1$ and $w_2$ are adjacent in $G\times vw_1$, and $(G\times vw_2)-v = ((G\times vw_1)-v)\times w_1w_2$; see~\cite{graphicIsoSystems}.

\subsection*{Cut-rank and rank-width}
Let $G$ be a graph with adjacency matrix $A$. That is, $A$ is the $V(G) \times V(G)$ matrix whose $(u,v)$ entry is one if $uv \in E(G)$ and zero otherwise. The \textit{cut-rank} of $X\subseteq V(G)$, denoted $\rho_G(X)$ (or just $\rho(X)$ if the graph is clear) is the rank over the binary field of the submatrix of $A$ with rows $X$ and columns $V(G)\setminus X$. As a function on subsets of $V(G)$, cut-rank is symmetric and submodular~\cite{RWAndVM}. Furthermore, the cut-rank function is invariant under local complementation:

\begin{lemma}\cite[Proposition~2.6]{RWAndVM}
\label{lemma:cutrank}
If $G$ and $\tilde{G}$ are locally equivalent and $X \subseteq V(G)$, then $\rho_G(X) = \rho_{\tilde G}(X)$.
\end{lemma}

We next define rank-width, which was introduced by Oum and Seymour~\cite{approxCWBW}. These definitions are not needed in the paper, but we include them for completeness. A \textit{rank-decomposition} of a graph $G$ is a tree $T$, having $V(G)$ as its set of leaves, whose vertices each have degree either one or three. The \textit{width} of an edge $e$ of $T$ is the cut-rank in $G$ of the set of all leaves of one of the components of $T-e$. Finally, the $\textit{rank-width}$ of $G$ is the minimum, over all rank-decompositions $T$ of $G$, of the maximum width of an edge of $T$. Graphs with at most one vertex do not admit rank-decompositions and we define their rank-width to be zero. It follows from Lemma~\ref{lemma:cutrank} that if $H$ is a vertex-minor of $G$, then the rank-width of $H$ is at most the rank-width of $G$. 

To prove Theorem~\ref{main} it suffices to consider a graph $G$ that is vertex-minor-minimal with rank-width at least $r(H)+1$. The following result of Oum~\cite{RWAndVM} shows that $G$ is highly connected in the sense that one side of any separation with low cut-rank is necessarily small. For a positive integer $m$ and a function $f$, a graph $G$ is \textit{$(m,f)$-connected} if for every partition $(X, Y)$ of $V(G)$ with $\rho(X)<m$, either $|X| \leq f(\rho(X))$ or $|Y|\leq f(\rho(X))$.

\begin{lemma}\cite[Lemma~5.3]{RWAndVM}
\label{lemma:mfConnected}
Define a function $g_{\ref{lemma:mfConnected}}:\bZ\rightarrow \bZ$ by $g_{\ref{lemma:mfConnected}}(n) = (6^{n}-1)/5$. For every positive integer $r$, if $G$ is a graph that is vertex-minor-minimal with rank-width at least $r$, then $G$ is $(r,g_{\ref{lemma:mfConnected}})$-connected.
\end{lemma}

There is an easy partial converse to Lemma~\ref{lemma:mfConnected} that, if $G$ is an $(r,g_{\ref{lemma:mfConnected}})$-connected graph with at least $3 g_{\ref{lemma:mfConnected}}(r-1)$ vertices,
then $G$ has rank-width at least $r$. It follows that, with respect to proving Theorem~\ref{main}, it suffices to consider large $(r,g_{\ref{lemma:mfConnected}})$-connected graphs, which is why we do not explicitly require the definition of rank-width. We do, however, require one additional result of 
Kwon and Oum~\cite{scattered}, on rank-width; in this result, by a {\em star} we mean a tree having at most one
non-leaf vertex.

\begin{theorem}\cite[Theorem~1.6]{scattered}
\label{thm:SF}
There is a function $r_{\ref{thm:SF}}:\bZ\rightarrow\bZ$ so that, for all positive integers $m$ and $k$, 
if $G$ is a graph of rank-width at least $r_{\ref{thm:SF}}(m,k)$, then $G$ has a vertex-minor with $m$ components 
each of which is a star on $k+1$ vertices.
\end{theorem}

%CIRCLE GRAPHS
\section{Circle Graphs}
\label{sec:circleGraphs}

A {\em chord diagram} is a collection of chords of the unit circle. A \textit{circle graph} is the intersection graph of chords in a chord diagram.
We allow two chords to have a common end on the circle, however, it is always possible to perturb the chords so as to avoid this; a chord diagram is {\em simple} if no two chords have a common end. 

 The main result of this section is that each circle graph is isomorphic to a vertex-minor of a comparability grid.
\begin{lemma}
\label{lemma:compGrid}
Every circle graph on $n$ vertices is isomorphic to a vertex-minor of the $3n \times 3n$ comparability grid.
\end{lemma}

To prove this result we show that every circle graph is a vertex-minor of a ``permutation graph" and that every permutation graph is an induced subgraph of a comparability grid. For a permutation $\pi$ of $\{1,\ldots,n\}$ the {\em permutation graph} represented by $\pi$ is the graph $F_{\pi}$ with vertex set $\{1,\ldots,n\}$ where vertices $i$ and $j$, with $i<j$, are adjacent if and only if $\pi_i>\pi_j$. To see that permutation graphs are circle graphs, place distinct points $b_1,b_2,\ldots,b_n,a_n,a_{n-1},\ldots,a_1$ in clockwise order around a circle and represent each vertex $i\in\{1,\ldots,n\}$ by the chord connecting $a_i$ to $b_{\pi_i}$.

\begin{figure}
\centering
\begin{tikzpicture}[scale = 1.5, every node/.style={inner sep=0,outer sep=2, fill=none}]
%Will draw chord diagram of k x n comparability grid
\draw[thick] (0,0) circle (2) {};
\foreach \i in {1,2,3}
{
    \filldraw
    (\i*40+100:2) circle (2pt)
    (180-100-40*\i:2) circle (2pt);
    \node[label=left:{$a_{\i}$}, draw] (A\i) at (\i*40+100:2) {};
    \node[label=right:{$b_{\i}$}, draw] (B\i) at (180-100-40*\i:2) {};
}
\foreach \i in {1,2,3}
{\foreach \j in {1,2,3}
    {
    \draw (A\i) -- (B\j);
    }
}
\end{tikzpicture}
\caption{A chord diagram for $F_3$}
\label{fig:chordCG}
\end{figure}
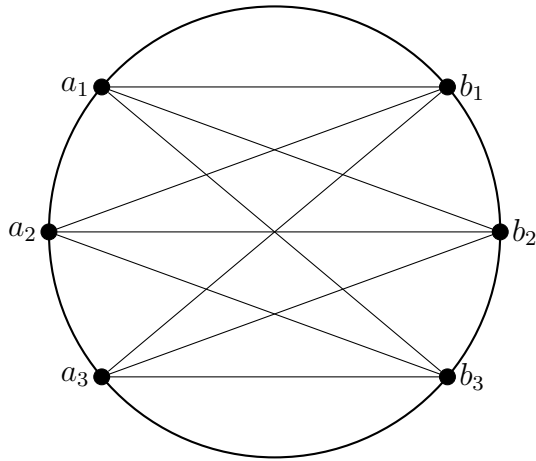

Let $\cC$ denote the set of all chords having one end in $\{a_1,\ldots,a_n\}$ and one end in $\{b_1,\ldots,b_n\}$ and let $F_n$
denote the corresponding circle graph. For example, the chord diagram for $F_3$ is depicted in Figure~\ref{fig:chordCG}. The chords in bold in Figure \ref{fig:chordISG} depict the chord diagram for the permutation graph $F_\pi$ where $\pi=(1)(3,2)$. Note that:
\begin{itemize}
\item[$(i)$]  every $n$-vertex permutation graph is isomorphic to an induced subgraph of $F_n$, and
\item[$(ii)$] $F_n$ is isomorphic to the $n\times n$ comparability grid (the vertex $(i,j)$ of the comparability
grid is associated with the chord $a_ib_{n+1-j}$).
\end{itemize}

Thus we have proved that:
\begin{lemma}
\label{splittable}
Every $n$-vertex permutation graph is isomorphic to an induced subgraph of the $n \times n$ comparability grid.
\end{lemma}
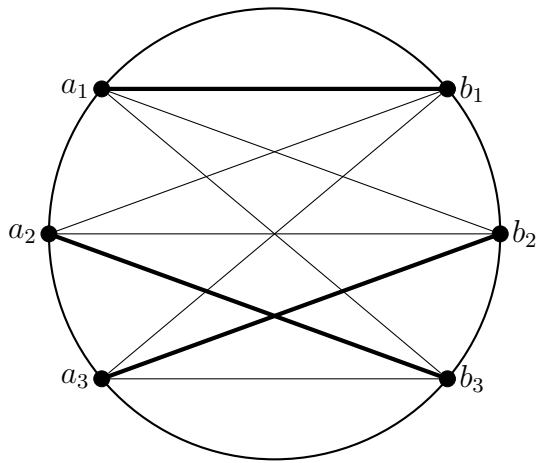
\begin{figure}
\centering
\begin{tikzpicture}[scale = 1.5, every node/.style={inner sep=0,outer sep=2, fill=none}]
%Will draw chord diagram of k x n comparability grid
\draw[thick] (0,0) circle (2) {};
\foreach \i in {1,2,3}
{
    \filldraw
    (\i*40+100:2) circle (2pt)
    (180-100-40*\i:2) circle (2pt);
    \node[label=left:{$a_{\i}$}, draw] (A\i) at (\i*40+100:2) {};
    \node[label=right:{$b_{\i}$}, draw] (B\i) at (180-100-40*\i:2) {};
}
\foreach \i in {1,2,3}
{\foreach \j in {1,2,3}
    {
    \draw (A\i) -- (B\j);
    }
}
 \draw[ultra thick] (A1) -- (B1);
 \draw[ultra thick] (A2) -- (B3);
 \draw[ultra thick] (A3) -- (B2);
\end{tikzpicture}
\caption{Chord diagrams for $F_\pi$ and $F_3$}
\label{fig:chordISG}
\end{figure}

The class of circle graphs is closed under vertex-minors and, to complete the proof of Theorem~\ref{lemma:compGrid},
we need to understand the effect of local complementation on chord diagrams. Let $\cC$ be a simple chord diagram 
for a circle graph $G$ and let $v\in V(G)$. The chord $v$ separates the circle into two open arcs $(A_1,\, A_2)$, and
we can obtain a chord diagram for $G*v$ by ``flipping'' $A_1$ (where by flip we mean invert the arc under reflective symmetry; chords follow their ends).

We conclude this section by proving the following result, which completes the proof of Lemma~\ref{lemma:compGrid}.
\begin{lemma}
\label{vmSplittable}
Every circle graph on $n$ vertices is a vertex-minor of  a permutation graph on $3n$ vertices.
\end{lemma}

\begin{proof}
\begin{figure}
%Draws chord diagram for C_1
\centering
\begin{tikzpicture}[scale = 1.5, every node/.style={inner sep=0,outer sep=2, fill=none}]
\def \anga {130} %angle for (a)
\def \angb {230} %angle for (b)
\def \angc {0} %angle for (c)
\def \far {11} %how far x and y are from (b)
\draw[thick] (0,0) circle (2) {};
\draw[ultra thick] (90:2) arc (90:270:2);
\node[label=left:{$A$}, draw] at (180:2) {};

\node[label=above left:{$a$}, draw] (a) at (\anga:2) {};
\filldraw (a) circle (2pt);
\node[label=below left:{$b$}, draw] (b) at (\angb:2) {};
\filldraw (b) circle (2pt);
\node[label=right:{$c$}, draw] (c) at (\angc:2) {};
\filldraw (c) circle (2pt);

\draw (a) -- (b) node[midway,label=left:{$v$}]{};
\draw (\angb-\far:2) -- (\angc+\far:2) node[midway,label=above:{$x$}]{};
\draw (\angb+\far:2) -- (\angc-\far:2) node[midway,label=below:{$y$}]{};
\end{tikzpicture}
\caption{The chord diagram $\cC_1$}
\label{fig:vm-circle}
\end{figure}
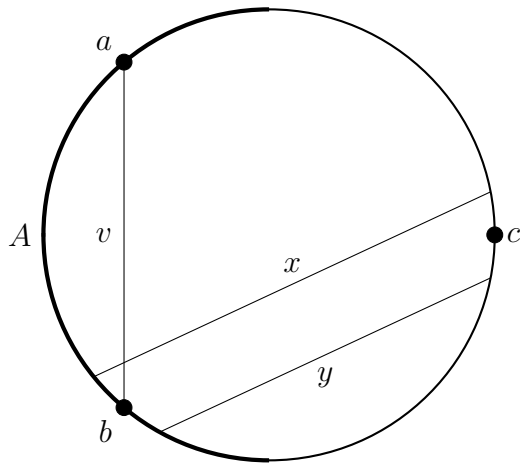

Consider a simple chord diagram $\cC$ for a circle graph $G$ and let $A$ be an arc of the unit circle
whose ends are disjoint from $\cC$. A chord is {\em crossing} if it has exactly one end in $A$. 
We may assume that there exist non-crossing chords in $\cC$ since otherwise 
$G$ is itself a permutation graph and the result follows easily. 
We will construct a chord diagram $\cC_2$ such that: 
\begin{itemize}
\item[$(i)$]  $|\cC_2| = |\cC|+2$,
\item[$(ii)$] $\cC_2$ has fewer non-crossing chords, and
\item[$(iii)$] the intersection graph of $\cC_2$ contains $G$ as a vertex-minor.
\end{itemize}
The result follows by iterated applications of this construction. 

Let $v\in \cC$ be a non-crossing chord with ends $a$ and $b$; we may assume that $a,b\in A$. Now select a point $c$ on the unit circle disjoint from $A$ and disjoint from $\cC$. Let $\cC_1$ be obtained from $\cC$ by adding two parallel chords $x$ and $y$ immediately on either side of the chord $[b,c]$, and let $\cC_2$ be obtained from $\cC_1$ by replacing the chord $v$ with the chord $[a,c]$. See Figure~\ref{fig:vm-circle}. Clearly $\cC_2$ satisfies $(i)$ and $(ii)$. Let $G_1$ and $G_2$ denote the intersection graphs of $\cC_1$ and $\cC_2$ respectively. Then $G_1$ is isomorphic to $G_2*x*y$ and $G$ is an induced subgraph of $G_1$. Thus $(iii)$ holds, as required.
\end{proof}

%CONNECTIVITY
\section{Connectivity}
\label{sec:connectivity}
In this section we review connectivity for vertex-minors and prove our ``Disentangling Lemma'', Lemma~\ref{shortPath}.

Let $S$ and $T$ be disjoint sets of vertices in a graph $G$ and let $A$ denote the adjacency matrix of $G$. The \textit{local connectivity of $S$ and $T$}, denoted by $\sqcap_G(S,T)$ (or simply $\sqcap(S,T)$), is the rank over the binary field of the submatrix $A[S,T]$. Notice that,
if $(S_1,\ldots,S_s)$ is a partition of $S$ and $(T_1,\ldots,T_t)$ is a partition of $T$, then
$$ \sqcap(S,T) \le \sum_{i=1}^s\sum_{j=1}^t \sqcap(S_i,T_j);$$
we refer to this property as {\em sub-additivity}. Moreover, since a rank-$k$ binary matrix has at most $2^k$ distinct columns, 
vertices in $S$ have at most $2^{\sqcap(S,T)}$ distinct neighbour sets in $T$. 

For a set $T\subseteq V(G)$, let $M_T$ denote the binary matroid represented by the submatrix $A[T, V\setminus T]$. Thus the ground set of $M_T$ is $V(G)\setminus T$ and a set $X\subseteq E(M_T)$ has rank $\sqcap(T,X)$. So a set $I\subseteq E(M_T)$ is independent 
if $|I| = \sqcap(T,I)$; we refer to the independent sets of $M_T$ as {\em $T$-independent} sets.

The \textit{connectivity between $S$ and $T$}, denoted by $\kappa_G(S,T)$ (or just $\kappa(S,T)$ when $G$ is clear from context), is the minimum of $\rho_G(X)$ over all sets $X \subseteq V(G)$ so that $S \subseteq X\subseteq V(G)\setminus T$.  Notice that if $G$ is $(m,f)$-connected, $t<m$, and both $S$ and $T$ have cardinality greater than $f(t)$, then $\kappa_G(S,T)>t$. The following is a version of Menger's Theorem for pivot-minors due to Oum~\cite{RWAndVM}; in essence the result states that two of the three ways of removing a vertex will preserve the connectivity between a pair of disjoint sets.

\begin{theorem}\cite[Lemma~4.4]{RWAndVM}
\label{lemma:waysToDelete}
Let $S$ and $T$ be disjoint sets of vertices in a graph $G$.
For every  $v\in V(G)\setminus (S \cup T)$ and $u \in N(v)$,
\begin{align*}
\kappa_G(S,T) &=\max \left(\kappa_{G-v}(S,T), \kappa_{(G*v)-v}(S,T) \right)\\
	&= \max \left(\kappa_{G-v}(S,T), \kappa_{(G \times uv)-v}(S,T) \right)\\
    &= \max \left(\kappa_{(G*v)-v}(S,T), \kappa_{(G \times uv)-v}(S,T) \right).
\end{align*}
\end{theorem}

Oum just stated the first two equalities, but the third follows from the first two by locally complementing at $v$ and then applying Lemma~\ref{lemma:vm} (which states that there are three ways to remove a vertex).

So, as proven by Oum~\cite[Theorem~6.1]{RWAndVM}, if $G$ is a graph and $S, T \subseteq V(G)$ are disjoint sets, then, by repeatedly applying Theorem~\ref{lemma:waysToDelete}, there is a pivot-minor $\tilde{G}$ of $G$ with $V(\tilde{G})=S \cup T$ so that $\sqcap_{\tilde{G}}(S,T) = \kappa_G(S,T)$. Taking a pivot-minor may change the edges inside $S \cup T$; that is, the graphs $\tilde{G}[S \cup T]$ and $G[S \cup T]$ may be different. The goal of our Disentangling Lemma is to still make the connectivity between $S$ and $T$ ``somewhat local'' but without changing the subgraph induced on $S \cup T$. The following definition formalizes what we mean by ``somewhat local''.

\begin{definition}[$k$-link]
\label{klink}
For a graph $G$ with disjoint $S, T \subseteq V(G)$, a $k$-{\em link for $(S,T)$} is a pair $(X_1, X_2)$ of $k$-element subsets of $V(G) \setminus (S \cup T)$ such that $X_1$ is $S$-independent, $X_2$ is $T$-independent, and either 
\begin{enumerate}
\item $X_1 = X_2$, or
\item $X_1$ and $X_2$ are disjoint, $\sqcap(X_1, X_2)=k$, all vertices in $X_1$ have the same set of neighbours in $T$, and all vertices in $X_2$ have the same set of neighbours in $S$.
\end{enumerate}
\end{definition}

We do not explicitly use the fact in the paper, but 
the motivation for $k$-links is that they certify high connectivity between $S$ and $T$; indeed,
if there exists a $k$-link for $(S,T)$, then $\kappa(S,T)\ge \frac 1 3 k$.
The Disentangling Lemma says that, if $\kappa(S,T)\gg k$, we can find a 
$k$-link in a locally equivalent graph without changing the induced subgraph on $S\cup T$. We would like to point out that the lemma would not hold if ``locally equivalent'' were replaced by ``equivalent up to pivoting''.

\begin{lemma}[Disentangling Lemma]
\label{shortPath}
There is a function $L_{\ref{shortPath}}:\bZ\rightarrow\bZ$ so that, for every positive integer $k$, if $G$ is a graph and  $S,T \subseteq V(G)$ are disjoint sets with $\kappa(S,T)\geq L_{\ref{shortPath}}(k)$ and $\sqcap(S,T)< k$, then there exists a graph $\tilde{G}$ that is locally equivalent to $G$ such that $\tilde{G}[S \cup T]=G[S \cup T]$ and $\tilde{G}$ has a $k$-link for $(S,T)$.
\end{lemma}

\begin{proof}Fix a positive integer $k$. Define $k_0 \coloneqq 2^{k-1}+1$ and
$$
L_{\ref{shortPath}}(k)\coloneqq 2^{k + k_0-2}  +2k-1.
$$
Suppose that the lemma fails for this function, and choose a counterexample $(G,S,T)$ with $|V(G)|$ minimum. 
We begin with two claims.

\begin{claim}
\label{sameNeighbourhood}
No two vertices in $V(G)\setminus (S \cup T)$ have the same set of neighbours in $S \cup T$.
\end{claim}
\begin{proof}
Suppose that $u,v \in V(G)\setminus (S \cup T)$ have the same set of neighbours in $S \cup T$. 
Consider the case that $uv\not\in E(G)$. Then $G[S\cup T]  = G*v*u[S\cup T]$, so $G[S\cup T]$ is a vertex-minor
of both $G-v$ and $(G*v)-v$. However, by Theorem~\ref{lemma:waysToDelete}, either
$\kappa_{G-v}(S,T)=\kappa_G(S,T)$ or $\kappa_{(G*v)-v}(S,T)=\kappa_G(S,T)$, contradicting the minimality of $G$. 
In the case that $uv\in E(G)$, we see that
$G[S\cup T]$ is a vertex-minor of both $G-v$ and $(G\times uv)-v$ and again we get a contradiction via Theorem~\ref{lemma:waysToDelete}.
\end{proof}

\begin{claim}
\label{neighbourhood}
There exist disjoint sets $Y_1, Y_2 \subseteq V(G)\setminus (S \cup T)$ so that
\begin{itemize} 
\item[$(i)$] $\sqcap(Y_1, Y_2)=|Y_1|=|Y_2|>2^{k_0-1}$, 
\item[$(ii)$] all vertices in $Y_1$ have the same set of neighbours in $T$, and 
\item[$(iii)$] all vertices in $Y_2$ have the same set of neighbours in $S$.
\end{itemize}
\end{claim}
\begin{proof}
Note that, if $X\subseteq V(G)\setminus(S\cup T)$ is a common independent set of $M_S\del T$ and $M_T\del S$ with cardinality $k$, then 
$(X,X)$ is a $k$-link for $(S,T)$; however, there is no such $k$-link, and hence $M_S\del T$ and $M_T\del S$
do not have a common independent set of size $k$.
So, by the Matroid Intersection Theorem,  there is a partition $(P,Q)$ of $V(G) \setminus (S \cup T)$ so that $\sqcap(S,P)+\sqcap(T,Q)<k$. 

Let $(P_1,\ldots,P_s)$ be the partition of $P$ into equivalence classes of identical columns of $A[S,P]$ and let
$(Q_1,\ldots,Q_t)$ be the partition of $Q$ into equivalence classes of identical columns of $A[T,Q]$. 
Since $\sqcap(S,P)+\sqcap(T,Q)\le k-1$, we have $st\le 2^{\sqcap(S,P)}2^{\sqcap(T,Q)}\le 2^{k-1}$. 

Note that $\sqcap(S\cup Q,\, T\cup P) \ge \kappa(S,T) \ge L_{\ref{shortPath}}(k) \ge s t 2^{k_0-1} +2k-1$.
Moreover, by sub-additivity and since $\sqcap(S,T)\le k-1$, 
\begin{eqnarray*}
\sqcap(Q,P) &\ge& \sqcap(S\cup Q,\, T\cup P) - \sqcap(S,T) - \sqcap(S,P) - \sqcap(Q,T) \\
&\ge& \sqcap(S\cup Q,\, T\cup P) - 2k+2 \\
&>& s t 2^{k_0-1}.
\end{eqnarray*}
So, again using sub-additivity, there exist $i\in\{1,\ldots,t\}$ and $j\in \{1,\ldots,s\}$ such that
$\sqcap(Q_i,P_j) > 2^{k_0-1}$. Now choose $Y_1\subseteq Q_i$ and $Y_2\subseteq P_j$ such that
$\sqcap(Y_1, Y_2)=|Y_1|=|Y_2|> 2^{k_0-1}$. Now it is straightforward to see that $(Y_1, Y_2)$ satisfies 
$(i)$, $(ii)$, and $(iii)$, as required.
\end{proof}

By Claim~\ref{sameNeighbourhood} and part $(ii)$ of Claim~\ref{neighbourhood}, no two vertices in $Y_1$ have the same set of neighbours in $S$. Then, since $|Y_1| > 2^{k_0-1}$, we have $\sqcap(Y_1,S) \ge k_0$. Let $Y'_1\subseteq Y_1$ be 
a $k_0$-element $S$-independent set. Since $Y_1$ is $Y_2$-independent, $Y'_1$ is also $Y_2$-independent.
So there is a $k_0$-element subset $Y'_2\subseteq Y_2$ that is $Y'_1$-independent. Now $|Y_2'|> 2^{k-1}$ so, by similar reasoning,
there exist a $k$-element subset $X_2\subseteq Y'_2$ that is $T$-independent and a $k$-element subset $X_1\subseteq Y'_1$ 
that is $X_2$-independent. Then $(X_1,X_2)$ is a $k$-link for $(S,T)$, a contradiction.
\end{proof}

%Ramsey theory
\section{Ramsey theory}
\label{sec:Ramsey}
The rest of this paper is dedicated to proving Theorem~\ref{mainGrid}, that every graph of sufficiently large rank-width has a vertex-minor isomorphic to the $n \times n$ comparability grid. For the proof it is convenient to work with graphs whose vertices are ordered.

An \textit{ordered set} is a sequence $X=(x_1, x_2, \ldots, x_k)$ with no repeated elements. A \textit{subset} of an ordered set $X$ is a subsequence of $X$. For the rest of this paper, the vertex set of every graph $G$ is an ordered set, and every set of vertices of $G$ is considered as an ordered subset of $V(G)$. Oftentimes this will not matter, but it will matter when we discuss disjoint sets $X$ and $Y$ of $V(G)$ which are ``coupled''. This will mean that the bipartite subgraph of $G$ which is induced between $X$ and $Y$ is one of a few specific graphs, like a perfect matching, where the vertices which are paired in the matching (for instance) are determined by the orderings of $X$ and $Y$. So we use the ordering of $V(G)$ to induce fixed orderings on subsets which do not depend on the particular coupled pair under consideration. We sometimes remind the reader of these conventions by writing that the ordering of $X$ is induced by the ordering of $V(G)$.

Furthermore, if $H$ is a subgraph of $G$, we mean that $V(H)$ is a subset of $V(G)$ as ordered sets. Two graphs are \textit{isomorphic} if they are isomorphic as graphs with unordered vertex sets. For each positive integer $n$, we fix a lexicographic ordering on the vertex set of the $n \times n$ comparability grid. The rest of the paper is dedicated to proving Theorem~\ref{mainGrid} with these conventions; this is easily seen to be equivalent to the original statement. 

Suppose $X$ and $Y$ are disjoint ordered sets of cardinality $k$ and $X' \subseteq X$. We write $\phiset{X}{Y}{X'}$ for the subset of $Y$ induced by the ordering of $X'$ with respect to $X$. That is, if $X=(x_1, x_2, \ldots, x_k)$, $Y=(y_1, y_2, \ldots, y_k)$, and $X'=(x_{i_1}, x_{i_2}, \ldots, x_{i_{k'}})$, then $\phiset{X}{Y}{X'} = (y_{i_1}, y_{i_2}, \allowbreak\ldots, y_{i_{k'}})$. We also write $\phiset{X}{X}{X'}$ for the set $X'$ itself.

In this section we review some Ramsey theory for graphs with ordered vertex sets.

For a graph $G$ with disjoint sets $X,Y \subseteq V(G)$, we say $X$ and $Y$ are \textit{anticomplete} if $G$ has no edges with one end in $X$ and one end in $Y$, and \textit{complete} if for all $x \in X$ and $y \in Y$, $xy \in E(G)$. We say $X$ and $Y$ are \textit{homogeneous} if they are either complete or anticomplete.

\begin{definition}
Let $X=(x_1, x_2, \ldots, x_k)$ and $Y=(y_1, y_2, \ldots, y_k)$ be disjoint sets of vertices in a graph $G$ with orderings induced by the ordering of $V(G)$. We say $(X,Y)$ is:
\begin{enumerate}
\item a {\em coupled matching} if  $N(x_i)\cap Y = (y_{i})$ for every $i \in \{1,2,\ldots, k\}$,
\item an {\em up-coupled half graph} if  $N(x_i)\cap Y = \allowbreak (y_{i}, y_{i+1}, \ldots, y_{k})$ for every $i \in \{1,2,\ldots, k\}$, and
\item a {\em down-coupled half graph} if  $N(x_{i})\cap Y = (y_{1}, y_{2}, \ldots, y_{i})$ for every $i \in \{1,2,\ldots, k\}$.
\end{enumerate}
\end{definition}

We say that $(X,Y)$ is the \textit{complement of a coupled matching} if $(X,Y)$ is a coupled matching in the complement of $G$. Similarly we will talk about the \textit{complement of a down-coupled half graph} and the \textit{complement of an up-coupled half graph}. If $(X,Y)$ is either a down-coupled half graph, an up-coupled half graph, or one of their complements, we say $(X,Y)$ is a \textit{coupled half graph}. If $(X,Y)$ is either a coupled matching, the complement of a coupled matching, or a coupled half graph, we say \textit{$X$ and $Y$ are coupled}. Notice that if $X$ and $Y$ are coupled and $X'\subseteq X$, then $X'$ and $\phiset{X}{Y}{X'}$ are coupled.

If $X$ and $Y$ are disjoint coupled sets in a graph $G$, then $\sqcap(X,Y)\geq |X|-1$. The next result,
due to Ding, Oporowski, Oxley, and Vertigan~\cite{DOVW}, shows that a partial converse holds; namely, that
if $\sqcap(X,Y)\gg k$, then, up to possibly reordering the vertices in $Y$, there are $k$-element subsets of $X$ and 
$Y$ that are coupled.

\begin{lemma}\cite[Theorem 2.3]{DOVW}
\label{cpl}
There is a function $R_{\ref{cpl}}:\bZ\rightarrow\bZ$ so that, for every positive integer $k$, if $G$ is a graph and  $X, Y \subseteq V(G)$ are disjoint sets with $\sqcap(X,Y)\geq R_{\ref{cpl}}(k)$, then there exist $k$-element subsets $X' \subseteq X$  and $Y'\subseteq Y$ that are coupled in a graph obtained from $G$ by reordering the vertices in $Y$.
\end{lemma}

\newtheorem*{ramseythm}{Ramsey's Theorem}

We use the following version of Ramsey's Theorem.
\begin{ramseythm}
\label{kRamsey} For each integer $k$, 
there is a function $R_{k}:\bZ\rightarrow\bZ$ so that, for  each positive integer $n$, every $k$-edge-coloured clique on at least $R_{k}(n)$ vertices contains a monochromatic clique of size $n$.
\end{ramseythm}

The following two results are easy applications of Ramsey's Theorem; we omit the proofs.
\begin{lemma}
\label{orderedRamsey}
There is a function $R_{\ref{orderedRamsey}}:\bZ\rightarrow\bZ$ so that, for every positive integer $k$, if $X$ and $Y$ are disjoint sets 
of vertices in a graph $G$ with $|X|=|Y| \geq R_{\ref{orderedRamsey}}(k)$, then there is a $k$-element subset $X'\subseteq X$ such that $X'$ and $\phiset{X}{Y}{X'}$ are either coupled or homogeneous.
\end{lemma}

\begin{lemma}
\label{bipRamsey}
There is a function $R_{\ref{bipRamsey}}:\bZ\rightarrow\bZ$ so that, for every positive integer $k$, if $X$ and $Y$ are disjoint sets
of vertices in a graph $G$ with $|X|,|Y|\ge R_{\ref{bipRamsey}}(k)$, then there exist $k$-element sets $X' \subseteq X$ and $Y' \subseteq Y$ such that $X'$ and $Y'$ are homogeneous.
\end{lemma}

For a function $R:\bZ\rightarrow\bZ$ and an integer $n>1$, we inductively define $R^{(n)}$ to be the function $R\circ R^{(n-1)}$, where for the base case $R^{(1)}=R$.

%Building a constellation
\section{Building a constellation}
\label{sec:bc}
Roughly speaking, a ``large constellation'' in a graph is an induced subgraph consisting of many large stars coupled together in a ``connected way". The proof of Theorem~\ref{mainGrid} then consists of two parts; in this section we prove that, up to local equivalence and reordering vertices, every graph of sufficiently large rank-width contains a large constellation. In the next section we prove that every graph containing a sufficiently large constellation has a vertex-minor isomorphic to the $n \times n$ comparability grid. 

Recall that a \textit{coclique} is a set of pairwise non-adjacent vertices.

\begin{definition}[Constellations] Let $G$ be a graph, let $n$ and $k$ be positive integers, and let $m$ be a non-negative integer. An {\em $(n,m,k)$-constellation in $G$} is a tuple $(H, (W_h:h \in H), K)$ such that
\begin{enumerate}
\item $H \subseteq V(G)$ is an $(n+m)$-vertex coclique,
\item the sets $(W_h:h \in H)$ are disjoint $k$-vertex cocliques in $G\del H$, with orderings induced by the ordering of $V(G)$,
\item $K$ is a connected $n$-vertex graph with  $V(K)\subseteq H$,
\item for every $h \in H$, the set $W_h$ is complete to $\{h\}$ and anticomplete to $H \setminus \{h\}$,
\item for distinct $u,v\in H$, the pair $(W_u,W_v)$ is either a coupled half graph or a coupled matching
if $uv\in E(K)$, and is anticomplete otherwise.
\end{enumerate}
\end{definition}
If $\C=(H, (W_h:h \in H), K)$ is an $(n,m,k)$-constellation in $G$, then
we write $H(\C)$ for $H$, we write $K(\C)$ for $K$, and for each $h \in H$ we write $\W{h}{\C}$ for $W_h$.
We denote the union of the sets $(\{v\}\cup W_v\, : \, v\in H)$ by $V(\C)$, 
we denote the union of the sets $(\{v\}\cup W_v\, : \, v\in V(K))$ by $A(\C)$, and we denote $V(\C)\setminus A(\C)$ by $B(\C)$.
We sometimes use a sequence of constellations $\C_0, \C_1, \ldots$, and in that case we write $\W{h}{0}$ for $\W{h}{\C_0}$, and likewise for $\C_1$, and so on. For $h \in H(\C)$ and $X \subseteq \W{h}{\C}$, we write $\C | X$ for $$\left(H, \left(\phiset{\W{h}{\C}}{\W{z}{\C}}{X}: z \in H \right), K\right).$$ Notice that  $\C|X$ is an $(n,m,|X|)$-constellation in $G$.

This section is devoted to proving that, for positive integers $n$ and $k$, every graph with sufficiently large rank-width contains, up to local equivalence and reordering vertices, an $(n,0,k)$-constellation. To build constellations we use  ``augmentations".
\begin{definition}[Weak augmentations]
For positive integers $n$, $m$, and $k$, a {\em weak $(n,m,k)$-augmentation} in a graph $G$ is a tuple
$(\C,x,y, X_1,X_2)$ such that $\C$ is an $(n,m,k)$-constellation; $x\in V(K(\C))$ and  $y\in H(\C)\setminus V(K(\C))$; and
$(X_1, X_2)$ is a pair of $k$-vertex subsets of $V(G)\setminus (V(\C))$, with orderings induced by the ordering of $V(G)$, such that $\W{x}{\C}$ and $X_1$ are coupled, $\W{y}{\C}$ and $X_2$ are coupled, and either
\begin{enumerate}
    \item $X_1 = X_2$, or
    \item $X_1$ and $X_2$ are disjoint and coupled, all vertices in $X_1$ have the same set of neighbours in $B(\C)$, and all vertices in $X_2$ have the same set of neighbours in $A(\C)$.
\end{enumerate}
\end{definition}

\begin{lemma}
\label{weakaug}
There is a function $k_{\ref{weakaug}}:\bZ^3\rightarrow\bZ$ so that, for all positive integers $n$, $m$, $k_0$, $k_1$, and $k_2$ with
$k_1\ge k_{\ref{weakaug}}(n,m,k_0)$ and $k_2\ge k_1$, if $\C$ is an $(n,m,k_2)$-constellation in a graph $G$ and
$\kappa(A(\C),B(\C))\ge k_1$, then there exists a graph that is equivalent to $G$ up to local complementation and reordering vertices and contains a weak $(n,m,k_0)$-augmentation.
\end{lemma}

\begin{proof} For positive integers $n$, $m$, and $k_0$ we define
\begin{align*}
t &\coloneqq n\left( mR_{\ref{cpl}}^{(3)}\left(k_0\right) + m  +1 \right), \mbox{ and}\\
k_{\ref{weakaug}}(n,m,k_0) &\coloneqq  L_{\ref{shortPath}}(t).
\end{align*}
Now let $k_1$ and $k_2$ be positive integers such that
$$k_2\ge k_1\ge k_{\ref{weakaug}}(n,m,k_0),$$
and let $\C$ be an $(n,m,k_2)$-constellation in a graph $G$ with
$\kappa(A(\C),B(\C))\ge k_1$.  By Lemma~\ref{shortPath},
there is a graph $G_0$ that is locally equivalent to $G$ such that
$G_0[V(\C)]= G [V(\C)]$, and
$G_0$ contains a $t$-link $(X_1, X_2)$ for $(A(\C),\allowbreak B(\C))$.
Up to local equivalence we may assume that $G_0=G$.

By sub-additivity, $\sqcap(A(\C)\setminus H(\C), X_1)\ge t-n$. Let 
$X_0\subseteq A(\C)\setminus H(\C)$ be a $(t-n)$-element $X_1$-independent set.
Now let $t' = mR_{\ref{cpl}}^{(3)}(k_0)+m$. Thus
$$|X_0| = t-n =  nt'.$$
 Thus, by the pigeonhole principle, there exist $x \in V(K(\C))$ and $X_0'\subseteq X_0$ of cardinality $t'$ so that $X_0' \subseteq \W{x}{\C}$. Note that $\sqcap(X_0', X_1)=|X_0'|$. By the definition of $X_1$ and $X_2$, there exist $X_1' \subseteq X_1$ and $X_2'\subseteq X_2$ so that $(X_1', X_2')$ is a $t'$-link for $(X_0', B(\C))$. By the same reasoning, there exist a vertex $y \in H(\C)\setminus V(K(\C))$ and $R_{\ref{cpl}}^{(3)}(k_0)$-vertex subsets $X_3'' \subseteq \W{y}{\C}$,  $X_2'' \subseteq X_2'$, and $X_1'' \subseteq X_1'$ so that $(X_1'', X_2'')$ is a $|X_3''|$-link for $(X_0', X_3'')$. 
 
Next, we apply Lemma~\ref{cpl} to the sets $X_0'$ and $X_1''$ so that, after possibly reordering the vertices in $X_1''$, there exist $R_{\ref{cpl}}^{(2)}(k_0)$-element subsets $Y_0 \subseteq X_0'$ and $Y_1 \subseteq X_1''$ so that $Y_0$ and $Y_1$ are coupled. The claim follows by repeating this process one or two more times depending on whether $X_1 = X_2$, and possibly reordering the vertices in $X_2''$ and $X_3''$. Note that it is fine to reorder vertices in $X_3''$ since $y \in H(\C)\setminus V(K(\C))$.
\end{proof}

When taking restrictions of a weak $(n,m,k)$-augmentation $(\C,x,y,X_1,X_2)$, we need to
respect orders between the sets $X_1$, $X_2$, and $(\W{h}{\C}\, : \, h\in V(K(\C))\cup\{y\})$, 
but not with the sets $(\W{z}{\C}\, : \, z\in H(\C)\setminus (V(K(\C))\cup\{y\}))$. To be more precise, consider
a $k'$-element subset $Y_1\subseteq X_1$.
Let $Y_2\coloneqq \phi_{X_1\rightarrow X_2}(Y_1)$ and let $\C'$ be an $(n,m,k')$-constellation such that
$H(\C')=H(\C)$; for each $h\in V(K(\C))\cup\{y\}$ we have $\W{h}{\C'} =  \phi_{X_1\rightarrow \W{z}{\C}}(Y_1)$;
and for each $z\in H(\C)\setminus (V(K(\C))\cup\{y\})$ the set $\W{z}{\C'}$ is a $k'$-element subset
of $\W{z}{\C}$. Then $(\C', x,y,Y_1,Y_2)$ is an $(n,m,k')$-augmentation.

\begin{definition}[Augmentations]
For positive integers $n$, $m$, and $k$, 
an {\em $(n,m,k)$-augmentation} is a weak $(n,m,k)$-augmentation $(\C,x,y,X_1,X_2)$ such that
for each $i\in \{1,2\}$:
\begin{enumerate}
\item $X_i$ is either a clique or a coclique,
\item for all $h \in V(K(\C))\cup\{y\}$, the sets $\W{h}{\C}$ and $X_i$ are either homogeneous or coupled, with orderings induced by the ordering of $V(G)$,
\item for all $h \in H(\C)\setminus (V(K(\C))\cup\{y\})$, the sets $\W{h}{\C}$ and $X_i$ are homogeneous, and
\item for all $h \in H(\C)$, the sets $\{h\}$ and $X_i$ are homogeneous.
\end{enumerate}
\end{definition}

\begin{lemma}
\label{aug}
There is a function $k_{\ref{aug}}:\bZ^3\rightarrow\bZ$ so that, for all positive integers $n$, $m$, $k_0$, and $k_1$ with
$k_1\ge k_{\ref{aug}}(n,m,k_0)$, if $G$ is a graph containing a weak $(n,m,k_1)$-augmentation, then $G$ contains an $(n,m,k_0)$-augmentation.
\end{lemma}

\begin{proof} For positive integers $n$, $m$, and $k_0$ we define
$$k_{\ref{aug}}(n,m,k_0) \coloneqq
 R_2^{(2)}\left(R_{\ref{orderedRamsey}}^{(2n)}\left(R_{\ref{bipRamsey}}^{(2m-2)}\left(k_0 \cdot 2^{2(m+n)}\right)\right) \right).$$
Now consider a weak $(n,m,k_1)$-augmentation $(\C,x,y,X_1,X_2)$ with $k_1 \ge k_{\ref{aug}}(n,m,k_0)$.

By applying Ramsey's Theorem first on $X_1$ and then on the specified subset of $X_2$, we can get statement (1) to hold. Now, for each $i\in\{1,2\}$ and $h\in V(K(\C))\cup \{y\}$ so that $\W{h}{\C}$ and $X_i$ are not already coupled, we successively apply Lemma~\ref{orderedRamsey} to get statement (2) to hold. Note that we apply the lemma at most $2n$ times since $\W{x}{\C}$ and $X_1$ are already coupled, as are $\W{y}{\C}$ and $X_2$. Then, for each $i\in\{1,2\}$ and for each $h\in H(\C)\setminus (V(K(\C))\cup\{y\})$, we successively apply Lemma~\ref{bipRamsey} to  get statement (3) to hold. Finally we get statement (4) to hold by, for each $i\in\{1,2\}$ and each $h\in H(\C)$, successively applying a majority argument to the edges from $h$ to what remains of $X_i$. 
\end{proof}

We can now prove the main result of this section.
\begin{lemma}
\label{strBC}
There is a function $r_{\ref{strBC}}:\bZ^3\rightarrow\bZ$ so that, for all positive integers $n$, $m$, and $k$, every graph of rank-width at least $r_{\ref{strBC}}(n,m,k)$ has a vertex-minor which contains an $(n,m,k)$-constellation, after possibly reordering vertices.
\end{lemma}
\begin{proof}
For $n=1$, the result is true by Theorem~\ref{thm:SF} with $r_{\ref{strBC}}(1,m,k) \coloneqq r_{\ref{thm:SF}}(m+1,k)$. Now assume that for some fixed integer $n \geq 2$, for all positive integers $m$ and $k$, such a function $r_{\ref{strBC}}(n-1,m,k)$ exists. Now, for fixed $m$ and $k$, we will show that $r_{\ref{strBC}}(n,m,k)$ exists. Define
\begin{align*}
k_1 &\coloneqq k_{\ref{weakaug}}\left(n-1,m+1,k_{\ref{aug}}(n-1,m+1,k+4)\right),\\
k_0 &\coloneqq g_{\ref{lemma:mfConnected}}(k_1), \textrm{ and}\\
r_{\ref{strBC}}(n,m,k) &\coloneqq \max \left( r_{\ref{strBC}}\big(n-1,m+1,k_0\big), k_1 \right).
\end{align*}

Toward a contradiction, suppose that $G$ is a graph with rank-width at least $r_{\ref{strBC}}(n,m,k)$ that does not have a vertex-minor containing an $(n,m,k)$-constellation, after possibly reordering vertices. Choose such a graph with $|V(G)|$ minimum; thus no proper vertex-minor of $G$ has rank-width at least
$r_{\ref{strBC}}(n,m,k)$. So, by Lemma~\ref{lemma:mfConnected}, the graph $G$ is
$(k_1,\, g_{\ref{lemma:mfConnected}})$-connected.

We may assume that $G$ contains an $(n-1, m+1, k_0)$-constellation $\C_0$.
Since $\min(|A(\C_0)|,\, |B(\C_0)|)\ge  g_{\ref{lemma:mfConnected}}(k_1)$, we have
$\kappa(A(\C_0),B(\C_0))\ge k_1$. Then, by Lemmas~\ref{weakaug} and~\ref{aug}, there 
is a graph equivalent to $G$ up to local complementation and reordering vertices
that contains an $(n-1,m+1,k+4)$-augmentation.

We choose a graph $G_1$ that is locally equivalent to $G$ and has an $(n-1,m+1, t)$-augmentation $(\C_1,z_1,z_2,Z_1,Z_2)$ such that:
\begin{enumerate}
\item either
\begin{itemize}
\item $Z_1=Z_2$ and $t=k+2$, or
\item $Z_1\neq Z_2$ and $t = k+4$,
\end{itemize}
\item subject to $(1)$ we have $Z_1=Z_2$ if possible, and
\item subject to $(2)$   the vertex
$z_2$ is complete to $Z_2$ if possible.
\end{enumerate}
We may assume that $G_1=G$.

\begin{claim}
\label{largeDegree}
There is a vertex in $\W{z_2}{1}\cup \{z_2\}$ with at least $t-1$ neighbours in either
$Z_1$ or  $Z_2$.
\end{claim}

\begin{proof}
Suppose otherwise, then, by the assumption,
\begin{itemize}
\item $z_2$ is anticomplete to $Z_1\cup Z_2$,
\item $(\W{z_2}{1},\, Z_2)$ is a coupled matching, and
\item if $Z_1\neq Z_2$, then $\W{z_2}{1}$ is anticomplete to $Z_1$.
\end{itemize}
Note that each vertex in $\W{z_2}{1}$ has degree $2$ in $G[V(\C_1)\cup Z_1\cup Z_2]$.
Let $G'$ be the graph obtained from $G$ by locally complementing on 
each vertex in $\W{z_2}{1}$. Note that $(\C_1,z_1,z_2,Z_1,Z_2)$
is an $(n-1,m+1, t)$-augmentation in $G'$ and $z_2$ is complete to $Z_2$ in $G'$,
contrary to our choice of $G_1$ and $(\C_1,z_1,z_2,Z_1,Z_2)$.
\end{proof}

We break the proof into two cases; there is a lot of overlap in the proofs, but it is less awkward
with the cases separated.

\medskip

\noindent 
{\bf Case 1:} {\em There is a vertex $v\in \W{z_2}{1}\cup \{z_2\}$ with at least $t-1$ neighbours in $Z_1$.}

\medskip

We choose $G_2\in \{G, \, G*v\}$ so that the set of neighbours of $v$ in $Z_1$ is a coclique in $G_2$. Let $w$ be the first vertex in $Z_1$ that is a neighbour, in $G_2$, of $v$, and let  $G_3\coloneqq G_2\times vw$. We will show that $G_3$ contains an $(n,m,k)$-constellation $\C_3$, giving a contradiction.

Let $H_3 \coloneqq (H(\C_1)\setminus\{z_2\})\cup \{w\}$, let $W^3_w$ denote a $k$-element  subset of the neighbours, in $G_2$, of $v$ in $Z_1\setminus \{w\}$, and, for each $x\in H(\C_1)\setminus\{z_2\}$, let $W^3_x \coloneqq \phi_{Z_1\rightarrow \W{x}{1}}(W^3_w)$. Note that since $w$ is the first neighbour of $v$ in $Z_1$, the vertex $w$ is either complete or anticomplete to each $W^3_x$ in $G_2$.

Now let $K_3$ denote the graph obtained from $K(\C_1)$ by adding the vertex $w$ and all edges $wx$ where $x\in V(K(\C_1))$ and $(W^3_w, W^3_x)$ is coupled; since $(W^3_w, W^3_{z_1})$ is coupled, $K_3$ is connected. Finally let $\C_3\coloneqq (H_3,(W^3_x\, : \, x\in H_3),K_3)$. We claim that $\C_3$ is an $(n,m,k)$-constellation in $G_3$ which follows from Lemma~\ref{lemma:pivot} and the following observations about adjacencies in $G_2$:
\begin{itemize}
\item $v$ is anticomplete to $V(\C_1)\setminus (\{z_2\}\cup \W{z_2}{1})$ and is complete to $W^3_w$,
\item for each $x\in H_3\setminus \{w\}$, the vertex $x$ is complete or anticomplete to $W^3_w\cup\{w\}$,
\item for each $x\in H_3\setminus V(K_3)$, the set $W_x^3$ is complete or anticomplete to $W^3_w\cup\{w\}$, and
\item for each $x\in V(K_3)$, the vertex $w$ is complete or anticomplete to $W^3_x$.
\end{itemize}

\medskip

\noindent 
{\bf Case 2:} {\em No vertex in $ \W{z_2}{1}\cup \{z_2\}$ has at least $t-1$ neighbours in $Z_1$.}

\medskip	

Then, by the above claim, there is a vertex $v\in \W{z_2}{1}\cup \{z_2\}$ with $t-1$ neighbours in $Z_2$. Thus $Z_2\neq Z_1$ and, by the definition of an augmentation, $v$ is anticomplete to $Z_1$.

We choose $G_2\in \{G, \, G*v\}$ so that the set of neighbours of $v$ in $Z_2$ is a coclique in $G_2$. Let  $w$ be the first neighbour, in $G_2$, of $v$ in $Z_2$ and let  $G_3\coloneqq G_2\times vw$. We will show that $G_3$ contains an $(n-1,m+1,k+2)$-augmentation $(\C_3,z_1,w,X,X)$ for some $\C_3$ and $X$, giving a contradiction to our choice of $G_1$ and $(\C_1,z_1,z_2,Z_1,Z_2)$.

Let $H_3 \coloneqq (H(\C_1)\setminus\{z_2\})\cup \{w\}$, let $W^3_w$ denote a $(k+2)$-element subset of the set of neighbours of $v$ in $Z_2\setminus \{w\}$, and, for each $x\in H(\C_1)\setminus\{z_2\}$, let $W^3_x \coloneqq \phi_{Z_2\rightarrow \W{x}{1}}(W^3_w)$. By the choice of $w$ to be the first neighbour of $v$ in $Z_2$, the vertex $w$ is either complete or anticomplete to each $W^3_x$ in $G_2$.

Finally let $\C_3\coloneqq (H_3,(W^3_x\, : \, x\in H_3),K(\C_1))$ and let $X \coloneqq \phi_{Z_2\rightarrow Z_1}(W^3_w)$. Again by the choice of $w$, the vertex $w$ is either complete or anticomplete to $X$ in $G_2$. We claim that $(\C_3,z_1,w,X,X)$ is an $(n-1,m+1,k+2)$-augmentation in $G_3$ which follows from Lemma~\ref{lemma:pivot} and the following observations about adjacencies in $G_2$:
\begin{itemize}
\item $v$ is anticomplete to both $V(\C_1)\setminus (\{z_2\}\cup \W{z_2}{1})$ and $X$, and is complete to $W^3_w$,
\item for each $x\in H_3\setminus \{w\}$, the vertex $x$ is complete or anticomplete to $W^3_w\cup\{w\}$,
\item for each $x\in H_3\setminus (V(K(\C_1))\cup \{w\})$, the set $W_x^3$ is complete or anticomplete to $W^3_w\cup\{w\}$, 
\item for each $x\in V(K(\C_1))$, the vertex $w$ is complete or anticomplete to $W^3_x$, and
\item $X$ is complete or anticomplete to $w$.
\end{itemize}
\end{proof}

%Extracting a Comparability Grid
\section{Extracting a comparability grid}
It remains to prove that every graph containing a sufficiently large constellation has a vertex-minor isomorphic to the $n \times n$ comparability grid. Henceforth we will only consider
$(n,m,k)$-constellations with $m=0$ and will abbreviate these to {\em $(n,k)$-constellations}.
 
We will apply the following well-known Ramsey-type lemma to  reduce to constellations whose 
associated graphs are stars, paths, or cliques.

\begin{lemma}
\label{induced}
There is a function $n_{\ref{induced}}:\bZ\rightarrow \bZ$ such that for every positive integer $k$, every connected graph on at least $n_{\ref{induced}}(k)$ vertices has a $k$-vertex induced subgraph that is either a path, a star, or a clique.
\end{lemma}

The following result gives a sufficient condition for a graph to contain arbitrary $n$-vertex graphs 
as vertex-minors.

\begin{lemma}
\label{sufficientPHG}
Let   $Z=(z_1, z_2, \ldots, z_n)$ be a set of vertices in a graph $G$, with ordering induced by the ordering of $V(G)$, so that there are distinct components $(A_{i,j}\, : \, 1\le i<j\le n)$ of $G-Z$ so that $z_i$ and $z_j$ have neighbours in $A_{i,j}$ and $N(V(A_{i,j}))\subseteq (z_i, z_{i+1}, \ldots, z_{j})$. Then every graph with vertex set $Z$  is a vertex-minor of $G$.
\end{lemma}
\begin{proof}
Let $H$ be a graph with vertex set $Z$. We say that a pair $(i,j)$, where  $1\leq i <j \leq n$, is {\em fixed} if for each $i'\leq i$ and $j'\geq j$ the vertices $z_{i'}$ and $z_{j'}$ are adjacent in either both of or neither of $H$ and $G$. If all edges are fixed then $H$ is an induced subgraph of $G$. Among all non-fixed pairs choose $(i,j)$ with $i$ minimum and, subject to that, $j$ is maximum. We will fix $(i,j)$, without unfixing any other pair, by locally complementing in $A_{i,j}$; the result 
follows by repeating this until all pairs are fixed.

There is an induced path $P=(v_1,v_2,\ldots,v_k)$ in $A_{i,j}$ such that $z_i$ is adjacent to $v_1$ but not to any of $v_2,\ldots,v_k$ and $z_j$ is adjacent to $v_k$ but not to any of $v_1,\ldots,v_{k-1}$ (if $z_i$ and $z_j$ share a neighbour then it is possible that $k=1$). Replacing $G$ with $G*v_1*v_2*\cdots *v_k$ fixes $(i,j)$ without unfixing any other pair, as required.
\end{proof}

The following two results are applications of Lemma~\ref{sufficientPHG} to constellations.
\begin{lemma}
\label{matchings}
For any $n$-vertex graph $H$, if $\C$ is an $\left(n, \binom{n}{2}\right)$-constellation in a graph $G$ such that $K(\C)$ is either a path or a clique, and for each edge $uv$ of $K(\C)$ the pair $(\W{u}{\C},\W{v}{\C})$ is a coupled matching, then $G$ has a vertex-minor isomorphic to $H$.
\end{lemma} 

\begin{proof}
We may assume that $V(G)=V(\C)$.
Let  $H(\C) =\{z_1,\ldots,z_n\}$ where, if $K(\C)$ is a path, then the vertices are in the order $(z_1,\ldots,z_n)$ on the path. Note that $G-H(\C)$ has $\binom n 2$ components which we label $(G_{i,j}\, : \, 1\le i<j\le n)$; each of these components is isomorphic to $K(\C)$. For each $1\le i<j\le n$, let $A_{i,j}$ denote the (unique) shortest path from the neighbour of $z_i$ in $G_{i,j}$ to the neighbour of $z_j$ in $G_{i,j}$. The result follows by applying Lemma~\ref{sufficientPHG} to the subgraph of $G$ induced on the union of $H(\C)$ together with the sets $(V(A_{i,j})\, : \, 1\le i<j\le n)$.
\end{proof}

\begin{lemma}[Star constellations]
\label{stars}
For any $n$-vertex graph $H$, if $\C$ is an $\left(\binom{n}{2}+1, n+2\right)$-constellation 
in a graph $G$ such that $K(\C)$ is a star,
then $G$ has a vertex-minor isomorphic to $H$.
\end{lemma}  

\begin{proof}
We may assume that $V(G)=V(\C)$. 
Let $H(\C) = \{h\}\cup \{v_{i,j}: \, 1\le i< j\le n\}$, where $h$ is the hub of the star $K(\C)$, and let  $\W{h}{\C} =(z_0,z_1,\ldots,z_n,z_{n+1})$. Note that, for each $1\le i<j\le n$, the graph $G[\W{v_{i,j}}{\C}\cup\{v_{i,j}\}]$ is a component of $G-(\W{h}{\C}\cup\{h\})$. By locally complementing and deleting vertices within the subgraph $G[\W{v_{i,j}}{\C}\cup\{v_{i,j}\}]$ we will obtain a connected graph $A_{i,j}$ such that $z_i$ and $z_j$ have neighbours in $A_{i,j}$ and $N(V(A_{i,j}))\cap\{z_1, \ldots, z_{n}\}\subseteq (z_i, z_{i+1}, \ldots, z_{j})$. Then the result will follow by applying Lemma~\ref{sufficientPHG} to the subgraph induced on the union of $\{z_1, \ldots, z_{n}\}$ and the sets $(V(A_{i,j})\, : \, 1\le i<j\le n)$.

In the case that $( \W{v_{i,j}}{\C},\W{h}{\C})$ is a coupled matching, we take $A_{i,j}$ to be the path in $G[\W{v_{i,j}}{\C}\cup\{v_{i,j}\}]$ connecting the neighbours of $z_i$ and $z_j$. Thus we may assume that $(\W{v_{i,j}}{\C},\W{h}{\C})$ is a coupled half graph. First suppose that $(\W{v_{i,j}}{\C},\W{h}{\C})$ is either a down-coupled half graph or the complement of an up-coupled half graph. Then, for each $k\in\{0,1,\ldots,n\}$, there is a vertex $x_k\in \W{v_{i,j}}{\C}$ whose neighbours in $\W{h}{\C}$ are $\{z_0,\ldots,z_k\}$. Let $G'= G\times v_{i,j}x_{i-1}$. Then, in $G'$, the set of neighbours of $x_j$ in $\{z_1,\ldots,z_n\}$ is $\{z_i,\ldots,z_j\}$, and we take $A_{i,j} = G'[x_j]$.

The final case that $(\W{v_{i,j}}{\C},\W{h}{\C})$ is either an up-coupled half graph or the complement of a down-coupled half graph is similar. In this case, for each $k\in\{1,\ldots,n+1\}$, there is a vertex $x_k\in \W{v_{i,j}}{\C}$ whose set of neighbours in $\W{h}{\C}$ is $\{z_k,\ldots,z_{n+1}\}$. We set $G'= G\times v_{i,j}x_{j+1}$ and then take $A_{i,j} = G'[x_i]$.
\end{proof}

Next we consider constellations whose associated graphs are cliques. In order to recognize comparability grids we use the following easy characterization.
\begin{figure}
\begin{tikzpicture}[scale=1.5, every node/.style={inner sep=0,outer sep=2, fill=none}]
%Will draw nxn comparability grid
\def \n {3}
\def \h {1.8}
\foreach \i in {1,...,\n} %unlabeled nodes
	{\foreach \j in {1,...,\n}{
  	\node (A\i\j) at (\h*\i,\h*\j) {};
  	\filldraw (A\i\j) circle (2pt);
}}
\node[label=left:{$(1,1)$}] at (A11) {};%label nodes
\node[label=left:{$(1,3)$}] at (A13) {};
\node[label=right:{$(3,1)$}] at (A31) {};
\node[label=right:{$(3,3)$}] at (A33) {};
\foreach \i in {1,...,\n}%straight edges
	{\foreach \l in {1,...,\n} {
	\foreach \j in {\i,...,\n}{
    	\foreach \r in {\l,...,\n}{
  			\draw[thick] (A\i\l) -- (A\j\r);
}}}}
\foreach \i in {1,...,\n}{%bendy edges
    \draw[thick] (A\i1) to [bend right=12] (A\i3);
    \draw[thick] (A1\i) to [bend right=12] (A3\i);
}
\draw[thick] (A11) to [bend right=10] (A33);%diagonal
\foreach \i in {1,...,\n}%bold a half graph
	{\foreach \j in {\i,...,\n} {
	    \draw[ultra thick] (A1\i) -- (A2\j);
}}
\end{tikzpicture}
\caption{The $3 \times 3$ comparability grid.}
\label{fig:comp-grid}
\end{figure}
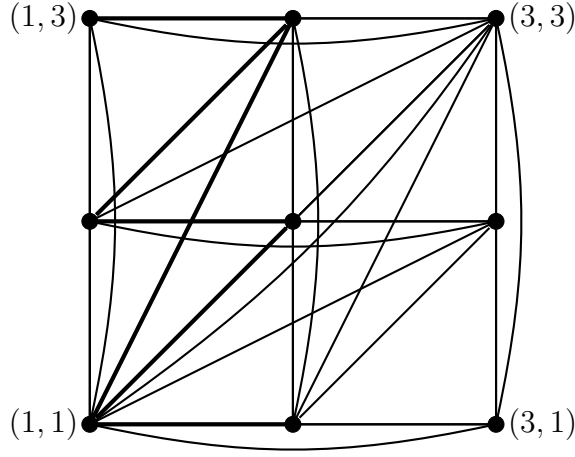

\begin{lemma}
\label{cliques1}
For any positive integer $n$, if $(X_1,\ldots,X_n)$ is a partition of the vertices of a graph $G$ into $n$-vertex cliques, with orderings induced by the ordering of $V(G)$, such that, for each $1\le i<j\le n$, the pair $(X_i,X_j)$ is an up-coupled half graph, then $G$ is isomorphic to the $n\times n$ comparability grid.
\end{lemma}

\begin{proof}
Recall that the $n\times n$ comparability grid has vertex set $\{(i,j): i,j \in \{1,2,\ldots, n\}\}$ where there is an edge between vertices $(i,j)$ and $(i',j')$ if either $i\leq i'$ and $j\leq j'$, or $i\geq i'$ and $j\geq j'$. Relabel the vertices of $G$ so that, for each $i\in \{1,\ldots,n\}$, we have $X_i = ((i,1),(i,2),\ldots,(i,n))$. Then $G$ is the $n\times n$ comparability grid. See Figure \ref{fig:comp-grid}, where the edges between $X_1$ and $X_2$ are bolded.
\end{proof}

\begin{lemma}
\label{cliques2}
For any positive integer $n$, if $(X_1,\ldots,X_{n^2})$ is a partition of the vertices of a graph $G$ into sets of cardinality $n^2$ such that, for each $1\le i<j\le n$, the pair $(X_i,X_j)$ is either an up-coupled half graph or the complement of a down-coupled half graph, then there is an induced subgraph of $G$ that is isomorphic to the $n\times n$ comparability grid.
\end{lemma}

\begin{proof}
Suppose that $X_i=(x_{i,1},\ldots,x_{i,n^2})$ for each $i\in\{1,\ldots,n^2\}$. Now, for each $i,j\in \{1,\ldots,n\}$, let $y_{i,j} \coloneqq x_{(i-1)n+j,(j-1)n+i}$ and let $Y_i=(y_{i,1},\ldots,y_{i,n})$. Thus $Y_1,\ldots, Y_n$ are cliques and, for each $1\le i<j\le n$, the pair $(Y_i,Y_j)$ is an up-coupled half graph, so the result follows from Lemma~\ref{cliques1}.
\end{proof}

\begin{lemma}[Clique constellations]
\label{cliques}
There are functions $n_{\ref{cliques}}:\bZ \rightarrow \bZ$ and $k_{\ref{cliques}}:\bZ \rightarrow \bZ$
such that,
for any positive integer $n$, if $\C$ is an $\left(n_{\ref{cliques}}(n),k_{\ref{cliques}}(n)\right)$-constellation 
in a graph $G$ such that $K(\C)$ is a clique,
then $G$ has a vertex-minor isomorphic to the $n\times n$ comparability grid.
\end{lemma}

\begin{proof}
Recall that the function $R_k$ is defined in Ramsey's Theorem.
For a positive integer $n$ we define
\begin{eqnarray*}
n_{\ref{cliques}}(n)&\coloneqq& R_3(n^2) \mbox{ and}\\
k_{\ref{cliques}}(n) &\coloneqq& \max\left(n^2,\binom{n^2}{2}\right).
\end{eqnarray*}
Let $\C$ be an $\left(n_{\ref{cliques}}(n),k_{\ref{cliques}}(n)\right)$-constellation 
in a graph $G$ such that $K(\C)$ is a clique and let
$H(\C)=(h_1,\ldots,h_{n_1})$, where $n_1 = n_{\ref{cliques}}(n)$.
Toward a contradiction we assume that  no vertex-minor of $G$ is isomorphic to the $n\times n$ comparability grid. 

By Ramsey's Theorem, there is a subsequence $(v_1,v_2,\ldots,v_{n^2})$ of
$(h_1,h_2,\ldots, h_{n_1})$ such that one of the following holds:
\begin{itemize}
\item[$(i)$] For each $1\le i<j\le n^2$, the pair $(\W{v_i}{\C},\W{v_j}{\C})$ is a coupled matching.
\item[$(ii)$] For each $1\le i<j\le n^2$, the pair $(\W{v_i}{\C},\W{v_j}{\C})$ is either an up-coupled half graph
or the complement of a down-coupled half graph.
\item[$(iii)$] For each $1\le i<j\le n^2$, the pair $(\W{v_i}{\C},\W{v_j}{\C})$ is either a down-coupled half graph or
the complement of an up-coupled half graph.
\end{itemize}
By possibly reversing the order of the sequence $(v_1,v_2,\ldots,v_{n^2})$ we may assume that
we are not in case $(iii)$. However,
Lemma~\ref{matchings} precludes case $(i)$ and Lemma~\ref{cliques2}
precludes case $(ii)$.
\end{proof}

It remains to consider constellations whose associated graphs are paths.
We say that a graph is an {\em ordered path} if the graph is a path and the 
order of the vertices on the path agrees with the ordering of the vertices of the graph; thus every path is 
isomorphic to an ordered path.
\begin{lemma}[Path constellations]
\label{paths}
There are functions $n_{\ref{paths}}:\bZ \rightarrow \bZ$ and $k_{\ref{paths}}:\bZ \rightarrow \bZ$
such that,
for any positive integer $n$, if $\C$ is an $\left(n_{\ref{paths}}(n),k_{\ref{paths}}(n)\right)$-constellation 
in a graph $G$ such that $K(\C)$ is a path,
then $G$ has a vertex-minor isomorphic to the $n\times n$ comparability grid.
\end{lemma}

\begin{proof}For a positive integer $n$ we define
\begin{eqnarray*}
m&\coloneqq& n^2, \\
k_3&\coloneqq& m\cdot 2^{m-1}, \\
k_2&\coloneqq& k_3+m-1, \\
k_1&\coloneqq& k_2+m-1, \\
n_{\ref{paths}}(n)&\coloneqq& (n^2-1)m, \mbox{ and}\\
k_{\ref{paths}}(n)&\coloneqq& \max\left(k_1, \binom{n^2}{2}\right).
\end{eqnarray*}
For convenience we also define $n_0 \coloneqq n_{\ref{paths}}(n)$ and $k_0 \coloneqq k_{\ref{paths}}(n)$. Let $\C$ be an $(n_0,k_0)$-constellation in a graph $G$ such that $K(\C)$ is an ordered path on vertices
$(h_1,\ldots,h_{n_0})$.
Toward a contradiction we may assume that  no vertex-minor of $G$ is isomorphic to the $n\times n$ comparability grid. 
\begin{claim} There is a graph $G_1$ that is locally equivalent to $G$ and has an $(m,k_1)$-constellation $\C_1$ 
such that $K(\C_1)$ is an ordered path with vertices
$(v_1,\ldots,v_{m})$ and, for each $i\in \{1,\ldots,m-1\}$ the pair $(\W{v_i}{1},\W{v_{i+1}}{1})$ is a coupled half graph.
\end{claim}

\begin{proof}
Let  $X$ denote the set of all $i\in\{1,\ldots,n_0-1\}$ such that $(\W{h_i}{\C},\W{h_{i+1}}{\C})$ is a coupled matching. Let $(v_1,\ldots,v_{t})$ be the restriction of the sequence $(h_1,\ldots,h_{n_0})$ to the elements $\{h_j\, : \, j\in \{1,\ldots,n_0\}\setminus X\}$. By Lemma~\ref{matchings}, the set $X$ cannot contain $n^2-1$ consecutive integers and hence $t\ge m$. Let $H_1 \coloneqq \{v_1,\ldots,v_{m}\}$, let $P_1$ be the ordered path on $(v_1,\ldots,v_{m})$, let $\C_1\coloneqq (H_1,(\W{v_1}{\C},\ldots,\W{v_{m}}{\C}),P_1)$, and let $G_1$ be the graph obtained from $G$ by locally complementing on each of the vertices in $(\W{h_i}{\C}\, : \, i\in X)$. It is routine to verify that the pair $(G_1,\C_1)$ satisfies the conclusion of the claim.
\end{proof}

Suppose that $A=(a_1,\ldots,a_l,a_{l+1})$ and $B=(b_1,\ldots,b_l,b_{l+1})$ are disjoint sets in a graph and $(A,B)$ is a coupled half graph. If $(A,B)$ is the complement of a down-coupled half graph then $((a_1,\ldots,a_l),(b_2,\ldots,b_{l+1}))$ is an up-coupled half graph, while, if $(A,B)$ is a down-coupled half graph, then $((a_1,\ldots,a_l),(b_2,\ldots,b_{l+1}))$ is the complement of an up-coupled half graph. Starting with the first elements of $\W{v_1}{1}$ and then choosing elements appropriately from each of $\W{v_2}{1},\ldots,\W{v_m}{1}$ in turn we obtain the following result.
\begin{claim}
\label{clm:upHalf} There is an $(m,k_2)$-constellation $\C_2$ in $G_1$
such that $K(\C_2)$ is an ordered path on vertices $(v_1, \ldots, v_m)$ and, for each $i\in \{1,\ldots,m-1\}$ the pair
$(\W{v_i}{2},\W{v_{i+1}}{2})$ is an up-coupled half graph or the complement of an up-coupled half graph.
\end{claim}

By pivoting we can further reduce to the case where all pairs are up-coupled half graphs.
\begin{claim}
There is a graph $G_3$ that is obtained from $G_1$ by pivoting and has an $(m,k_3)$-constellation $\C_3$ such that $K(\C_3)$ is an ordered path on vertices $(u_1, u_2, \ldots, u_m)$ and, for each $i\in \{1,\ldots,m-1\}$ the pair $(\W{u_i}{3},\W{u_{i+1}}{3})$ is an up-coupled half graph.
\end{claim}

\begin{proof}
We will prove by induction on $m-t$, where $1 \leq t\leq m$, that if a graph $G$ contains an $(m, k_3+m-t)$-constellation $\C_2$ such that $K(\C_2)$ is an ordered path on vertices $(w_1, w_2, \ldots, w_m)$ and for each $i \in \{1, \ldots, t-1\}$, the pair $(\W{w_i}{2},\W{w_{i+1}}{2})$ is an up-coupled half graph, and for each $i \in \{t, \ldots, m-1\}$, the pair $(\W{w_i}{2},\W{w_{i+1}}{2})$ is an up-coupled half graph or the complement of an up-coupled half graph, then there is a graph $G_3$ that is obtained from $G$ by pivoting and has an $(m,k_3)$-constellation $\C_3$ as in the claim. The case where $t=1$ implies the claim since $k_2 = k_3+m-1$. The base case where $t=m$ holds by deleting excess vertices from each set $\W{w_i}{2}$.

Now we may assume that $t<m$. We may also assume that the pair $(\W{w_t}{2},\W{w_{t+1}}{2})$ is the complement of an up-coupled half graph, as otherwise we may delete one vertex from each set $\W{w_i}{2}$ and apply induction. Let $w$ be the first vertex in $W_{w_{t+1}}^2$. Let $G_3 = G \times ww_{t+1}$, let $H_3 = (H(\C_2)\setminus \{w_{t+1}\})\cup \{w\}$, and let $K_3$ be the graph obtained from $K(\C_2)$ by relabeling $w_{t+1}$ to $w$. Let $W_w^3$ be the set obtained from $W_{w_{t+1}}^{2}$ by deleting $w$, and, for each $h \in H(\C_2)\setminus \{w_{t+1}\}$, let $W_h^3$ be the set obtained from $W_h^{2}$ by deleting its first vertex. Finally, let $\C_3 = (H_3, (W_h^3:h \in H_3), K_3)$.

Consider the neighbours of $w$ and $w_{t+1}$ in $G[V(\C_3)\cup \{w_{t+1}\}]$. The neighbourhood of $w_{t+1}$ is exactly $W_{w_{t+1}}^2$. The vertex $w$ is complete to $W_{w_t}^3$ and either complete or anticomplete to $W_{w_{t+2}}^3$, if $t+2\le m$. These are the only neighbours of $w$ other than $w_{t+1}$. Thus $\C_3$ is an $(m, k_3+m-t-1)$-constellation in $G_3$ so that all pairs are coupled in the same way as in $G$, except for $(W_{w_t}^3, W_w^3)$, which is an up-coupled half graph in $G_3$, and $(W_{w}^3, W_{w_{t+2}}^3)$, which may be complemented. The claim follows by the induction hypothesis.
\end{proof}

For each $s\in\{1,\ldots,m\}$, we let $L_s$ denote the graph with vertex set $\{u_1,\ldots,u_{m}\}$ and edge set 
$$\{u_iu_j\, : \, 1\le i<j\le s\}\cup\{u_{s}u_{s+1},
u_{s+1}u_{s+2},\ldots,u_{m-1}u_{m}\}.$$
Thus $L_1$ is a path and $L_{m}$ is a complete graph. For each $s\in \{1,\ldots,m\}$ we let $d_s\coloneqq m 2^{m-s}$; thus $d_1=k_3$ and $d_m=m=n^2$.

\begin{claim} For each $s\in\{1,\ldots,m\}$, there is a graph $G'_s$ that is locally equivalent to $G$ and has disjoint $d_s$-vertex cocliques $(X^s_1,\ldots,X^s_m)$ such that
\begin{itemize}
\item[$(i)$] for each $i\in\{1,\ldots, m\}$, $X^s_{i}\subseteq W^3_{u_i}$,
\item[$(ii)$] for $1\le i<j\le m$, the pair $(X^s_i,X^s_j)$ is an up-coupled half graph if $u_iu_j\in E(L_s)$ 
and is anticomplete otherwise, and
\item[$(iii)$] for each $i\in\{s+1,\ldots, m\}$, the vertex $u_i$ is complete to $X^s_i$ and
anticomplete to each of $X^s_1,\ldots,X^s_{i-1}$ and to each $X^s_{i+1}\cup \{u_{i+1}\},\ldots,X^s_m \cup \{u_m\}$.
\end{itemize}
\end{claim}

\begin{proof}
The proof is by induction on $s$; when $s=1$ the conclusion is satisfied by $G'_1 \coloneqq G_3$ and $X^1_i = \W{u_i}{3}$ for each $i\in\{1,\ldots,m\}$. For some $s\in\{2,\ldots,m\}$ suppose that there exist $G'_{s-1}$ and $(X^{s-1}_1, \ldots,X^{s-1}_m)$ as claimed; we will determine $G'_s$ and $(X^s_1,\ldots,X^s_m)$.

We let $G'_s$ be the graph obtained from $G'_{s-1}$ by locally complementing on each vertex in $X_{s-1}^{s-1}\cup \{u_s\}$. Suppose that, for each $i\in \{1,\ldots,m\}$, we have $X^{s-1}_i =(x^i_1,\ldots,x^i_{d_{s-1}})$, and let $X^s_i \coloneqq (x^i_1,x^i_3,\ldots, x^i_{(2d_s)-1})$. We claim that $G'_s$ and $(X^s_1,\ldots,X^s_m)$ satisfy the result; this follows from the following observations about adjacencies in $G'_{s-1}$:
\begin{itemize}
\item for each $i,j\in\{1,3,\ldots,d_{s-1}-1\}$ and each $a,b\in\{1,\ldots,s-2\}$, the vertices
$x^a_i$ and $x^b_j$ have an even number of common neighbours in $X^{s-1}_{s-1}\cup\{u_s\}$,
\item for each $i,j\in\{1,3,\ldots,d_{s-1}-1\}$ and each $a,b\in\{s,s+1,\ldots,m\}$, the vertices
$x^a_i$ and $x^b_j$ have an even number of common neighbours in $X^{s-1}_{s-1}\cup\{u_s\}$, and
\item for each $i,j\in\{1,3,\ldots,d_{s-1}-1\}$ and each $a\in\{1,\ldots,s-2\}$, the vertices
$x^a_i$ and $x^s_j$ have an odd number of common neighbours in $X^{s-1}_{s-1}\cup\{u_s\}$ if and only if $j\ge i$.
\end{itemize}
\end{proof}

We obtain the final contradiction to Lemma~\ref{paths} by applying Lemma~\ref{cliques2} to $G'_{m}$ and
$(X^m_1,\ldots,X^m_m)$.
\end{proof}

We can now combine the above results to prove our main result, Theorem~\ref{mainGrid},
which we restate here for convenience.
\mainG*

\begin{proof}
For a positive integer $n$ we define
\begin{eqnarray*}
k_1&\coloneqq& \max\left(n^2+2, k_{\ref{cliques}}(n), k_{\ref{paths}}(n) \right),\\
n_1&\coloneqq& \max\left(\binom{n^2}{2}+1, n_{\ref{cliques}}(n), n_{\ref{paths}}(n) \right), \mbox{ and}\\
f(n) &\coloneqq& r_{\ref{strBC}}(n_{\ref{induced}}(n_1),0,k_1).
\end{eqnarray*}
Let $G$ be a graph with rank-width at least $f(n)$.
By Lemmas~\ref{strBC} and~\ref{induced}, there is a graph $G_1$, equivalent to
$G$ up to local complementation and reordering vertices, that contains
an $(n_1,k_1)$-constellation $\C$ such that $K(\C)$ is either a star, a clique, or a path.
Now the result follows by Lemmas~\ref{stars},~\ref{cliques}, and~\ref{paths}.
\end{proof}

\section*{Acknowledgement}
We would like to thank the anonymous referees for carefully reading the paper and suggesting a number of helpful clarifications and corrections.
 
\bibliographystyle{plain}

\end{document}